\documentclass[11pt,a4paper]{article}

\usepackage[T1]{fontenc}
\usepackage[utf8]{inputenc}
\usepackage{newtxtext}
\usepackage{mathtools,amssymb,amsthm}
\usepackage{newtxmath}
\usepackage[final]{microtype}
\usepackage[
  a4paper,
  top=23mm,
  bottom=24mm,
  inner=24mm,
  outer=24mm,
  headheight=22pt
]{geometry}
\usepackage[dvipsnames]{xcolor}
\usepackage{setspace}
\usepackage{enumitem}
\usepackage{titlesec}
\usepackage{fancyhdr}
\usepackage{fancyvrb}
\usepackage{booktabs}
\usepackage{hyperref}
\usepackage{bookmark}
\usepackage[nameinlink,noabbrev]{cleveref}

\definecolor{ink}{HTML}{20252B}
\definecolor{accent}{HTML}{315A74}
\definecolor{rulegray}{HTML}{AAB4BC}
\definecolor{questionblue}{HTML}{3E78A8}
\definecolor{questionback}{HTML}{F4F8FC}
\definecolor{answergreen}{HTML}{4F8A63}
\definecolor{answerback}{HTML}{F4FAF5}
\definecolor{aiyellow}{HTML}{D3A13A}
\definecolor{aiback}{HTML}{FFF9E9}
\definecolor{aititle}{HTML}{F2CD72}
\color{ink}

\hypersetup{
  colorlinks=true,
  linkcolor=accent,
  citecolor=accent,
  urlcolor=accent,
  pdftitle={On Some More Problems from the Kourovka Notebook},
  pdfauthor={Vasily Ionin and Artem Semidetnov}
}

\setlist[enumerate]{
  label=\textcolor{accent}{\arabic*.},
  leftmargin=2.15em,
  itemsep=0.35em,
  topsep=0.35em
}

\titleformat{\section}
  {\large\bfseries\color{accent}}
  {\thesection}{0.65em}{}
\titleformat{\subsection}
  {\normalsize\bfseries\color{ink}}
  {\thesubsection}{0.65em}{}
\titlespacing*{\section}{0pt}{3.2ex plus 0.8ex minus 0.4ex}{1.0ex}
\titlespacing*{\subsection}{0pt}{1.8ex plus 0.5ex minus 0.2ex}{0.6ex}

\renewcommand{\headrulewidth}{0.35pt}
\renewcommand{\headrule}{\hbox to\headwidth{\color{rulegray}\leaders\hrule height \headrulewidth\hfill}}

\newtheorem{theorem}{Theorem}[section]
\newtheorem{proposition}[theorem]{Proposition}
\newtheorem{lemma}[theorem]{Lemma}

\theoremstyle{definition}

\theoremstyle{remark}

\newcommand{\statementbox}[3]{%
  \par\addvspace{1.1em}%
  \begingroup
  \setlength{\fboxsep}{8pt}%
  \setlength{\fboxrule}{0.8pt}%
  \noindent\fcolorbox{#1}{#2}{%
    \begin{minipage}{\dimexpr\linewidth-2\fboxsep-2\fboxrule\relax}
      #3
    \end{minipage}}%
  \par\endgroup
  \addvspace{1.1em}%
}

\NewDocumentEnvironment{Question}{o +b}{%
  \statementbox{questionblue}{questionback}{%
    \refstepcounter{theorem}%
    \noindent{\bfseries Question~\thetheorem
      \IfValueT{#1}{\ {\normalfont(#1)}}.}\enspace
    #2}%
}{}

\NewDocumentEnvironment{Answer}{o +b}{%
  \statementbox{answergreen}{answerback}{%
    \refstepcounter{theorem}%
    \noindent{\bfseries Answer~\thetheorem
      \IfValueT{#1}{\ {\normalfont(#1)}}.}\enspace
    #2}%
}{}

\NewDocumentEnvironment{aipractices}{+b}{%
  \refstepcounter{subsection}%
  \addcontentsline{toc}{subsection}{%
    \protect\numberline{\thesubsection}AI best practices}%
  \par\addvspace{2.0em}%
  \begingroup
  \setlength{\fboxsep}{8pt}%
  \setlength{\fboxrule}{0.8pt}%
  \noindent\fcolorbox{aiyellow}{aiback}{%
    \begin{minipage}{\dimexpr\linewidth-2\fboxsep-2\fboxrule\relax}
      \begingroup
      \setlength{\fboxsep}{4pt}%
      \noindent\colorbox{aititle}{%
        \parbox{\dimexpr\linewidth-2\fboxsep\relax}{%
          \bfseries\color{ink}\thesubsection\quad AI best practices}}%
      \endgroup
      \medskip
      #1
    \end{minipage}}%
  \par\endgroup
  \addvspace{2.3em}%
}{}

\makeatletter
\renewenvironment{abstract}{%
  \small\quotation
  \setlength{\leftskip}{0.8em}%
  \setlength{\rightskip}{0.8em}%
  \noindent
  {\bfseries\color{accent}Abstract.}\enspace\ignorespaces
}{%
  \endquotation
}
\makeatother

\begin{document}

\thispagestyle{plain}
\begin{center}
  {\fontsize{18}{21}\selectfont\bfseries
    On Some More Problems from the Kourovka Notebook\par}
  \vspace{0.45em}
  \begin{minipage}[t]{0.47\textwidth}
    \centering
    {\normalsize Vasily Ionin\par}
    \vspace{0.12em}
    {\footnotesize
      Saint Petersburg Department of\par
      Steklov Mathematical Institute\par
      27 Fontanka, St.~Petersburg, 191023, Russia\par
      \href{mailto:ionin.code@gmail.com}{\nolinkurl{ionin.code@gmail.com}}\par}
  \end{minipage}\hfill
  \begin{minipage}[t]{0.47\textwidth}
    \centering
    {\normalsize Artem Semidetnov\par}
    \vspace{0.12em}
    {\footnotesize
      Universit\'{e} de Gen\`{e}ve, Section de math\'{e}matiques\par
      rue du Conseil-G\'{e}n\'{e}ral 7--9, 1205 Gen\`{e}ve, Switzerland\par
      \href{mailto:artemsemidetnov@gmail.com}{\nolinkurl{artemsemidetnov@gmail.com}}\par
      \href{mailto:Artem.Semidetnov@etu.unige.ch}{\nolinkurl{Artem.Semidetnov@etu.unige.ch}}\par}
  \end{minipage}
\end{center}

\begin{abstract}
We give solutions to several problems in combinatorial group theory recorded
as open in the literature.  Our main source is the Kourovka Notebook, although
we also consider questions from elsewhere.  All solutions were found with the
assistance of large language models and checked by the authors.
\end{abstract}

\vspace{-0.3em}

\makeatletter
\noindent\hspace*{0.09\textwidth}%
\begin{minipage}{0.82\textwidth}
\begingroup
\small
\setcounter{tocdepth}{1}
\renewcommand*\l@section{\@dottedtocline{1}{0em}{1.6em}}
\tableofcontents
\endgroup
\end{minipage}
\par
\makeatother

\section*{Introduction}
\phantomsection
\addcontentsline{toc}{section}{Introduction}

This paper presents solutions to problems explicitly identified as open in the
literature.  Most are drawn from the Kourovka Notebook \cite{kourovka21}, but
some originate elsewhere.  The title of this paper echoes
\cite{vandoorn-et-al}, a recent paper on open problems in group theory that
attracted attention for its use of language models.

Every solution presented here was found with the assistance of artificial
intelligence.  The systems used were Codex with \mbox{ChatGPT~5.6~Sol} and
Claude Desktop with \mbox{Opus~5}.  They suggested proof strategies, located relevant results,
designed searches, and tested candidate constructions.  The authors then
checked every argument and computation independently.

\subsection*{Motivation}

Our motivation is twofold.

\begin{enumerate}
  \item We hope these solutions will be useful to researchers who considered
  the same questions without resolving them.  In several cases, the solution
  requires a finite computer search, an appropriate reference, or a short
  structural argument.  Publishing verified solutions may also improve the mathematical
  corpus on which later reasoning systems are trained.

  \item While searching for the solutions, we developed and compared workflows
  for interacting with large language models.  Some proved substantially more
  productive than others.  In the relevant sections, we state explicitly which
  practices were useful.
\end{enumerate}

We would be pleased if any of these solutions were to prompt further work or
advance an existing line of research.

\subsection*{Organization}

Each section treats one problem and may be read as a self-contained research
note: it gives the necessary background, states the question, and supplies
either a proof or an explicit construction with verification.  Every solution
has been checked manually, and the authors take responsibility for its
correctness.  The paper is organized as follows.

\begin{itemize}[
  label=\textcolor{accent}{\textbullet},
  leftmargin=1.45em,
  itemsep=0.15em,
  topsep=0.3em
]
  \item In Section~\ref{sec:problem-14-85}, we answer Problem~14.85
  \cite{kourovka21} affirmatively.
  \item In Section~\ref{sec:problem-19-94}, we answer Problem~19.94
  \cite{kourovka21} negatively.
  \item In Section~\ref{sec:problem-16-11}, we answer Problem~16.11
  \cite{kourovka21} negatively.
  \item In Section~\ref{sec:problem-17-47}, we answer Problem~17.47
  \cite{kourovka21} negatively.
  \item In Section~\ref{sec:question-2-8-14}, we answer Question~2.8.14
  \cite{LintonNybergBrodda2025} negatively.
  \item In Section~\ref{sec:neshchadim}, we answer Problem~18
  \cite{BardakovEtAl2015} affirmatively.
  \item In Section~\ref{sec:plotkin-stable-representations}, we answer
  Problem~4.3.5 \cite{Plotkin1977} affirmatively.
  \item In Section~\ref{sec:problem-17-32}, we answer
  Problem~17.32 \cite{kourovka21} negatively.
\end{itemize}

\subsection*{Acknowledgements}

The authors warmly thank Roman Mikhailov for his support and valuable
discussions.

\vspace{1.6\baselineskip}
We write \(G'=[G,G]\) for the \emph{commutator subgroup} (or \emph{derived
subgroup}) of a group \(G\).

\section{Primitive-preserving endomorphisms of free metabelian groups}
\label{sec:problem-14-85}

Let \(F_n\) be the free group of rank \(n\), and let

\[
M_n=F_n/F_n''
\]
be the \emph{free metabelian group} of rank \(n\).  An element is called
\emph{primitive} if it belongs to a free basis.  An endomorphism is
\emph{primitive-preserving} if it maps every primitive element to a primitive
element.

For free groups, every primitive-preserving endomorphism is an automorphism
\cite{Lee2002}.  For free metabelian groups, the same conclusion was known for
\(n\leq2\) \cite{GuptaTimoshenko1997}.  It was also known in every finite rank
under the stronger assumption that the endomorphism maps every tuple
extendable to a free basis to another such tuple \cite{Timoshenko2015}.  The
single-element question remained open.

\begin{Question}[{\cite[Problem~14.85]{kourovka21}}]
Suppose that an endomorphism \(\varphi\colon M_n\to M_n\) maps every primitive
element to a primitive element.  Must \(\varphi\) be an automorphism?
\end{Question}

This question also appears as Problem~(M0) in the online problem list
\cite{TimoshenkoShpilrain2026}.

\begin{Answer}
\label{ans:problem-14-85}
Yes.  Every primitive-preserving endomorphism of \(M_n\) is an automorphism.
\end{Answer}

\subsection{Solution}

The proof has three steps.
\begin{enumerate}[
  label=\textcolor{accent}{\arabic*.},
  leftmargin=2.15em,
  itemsep=0.15em,
  topsep=0.25em
]
  \item First we show that \(\varphi^{\mathrm{ab}}\) is an
  automorphism.  After composing with an automorphism induced from \(F_n\), we
  reduce to an endomorphism \(\psi\) satisfying
  \(\psi^{\mathrm{ab}}=\operatorname{id}\), with \(\psi\) an automorphism if
  and only if \(\varphi\) is.

  \item Fox identities reduce the problem to proving that the
  Jacobian \(J(\psi)\) is invertible.

  \item If \(\det J(\psi)\) were not a unit, reduction modulo a
  maximal ideal containing it would produce a nonzero left kernel.  A suitable
  primitive element has a specialized Fox row in this kernel, forcing its image
  under \(\psi\) to have a non-unimodular Fox row and hence to be
  nonprimitive.
\end{enumerate}

Fix a free basis \(x_1,\ldots,x_n\) of \(M_n\).  Put

\[
M_n^{\mathrm{ab}}=M_n/M_n'=F_n^{\mathrm{ab}},\qquad
R=\mathbb Z[M_n^{\mathrm{ab}}]
=\mathbb Z[t_1^{\pm1},\ldots,t_n^{\pm1}],\qquad
b=(t_1-1,\ldots,t_n-1)^{\mathsf T}
\in R^{n\times1},
\]
where \(t_i\) is the image of \(x_i\) in \(M_n^{\mathrm{ab}}\).  The Fox derivatives on
\(\mathbb Z[F_n]\), with respect to the basis
\(x_1,\ldots,x_n\), are recalled in \cite{Fox1953,MKS1976}.  For
\(u\in M_n\), write

\[
\partial_i u:=\overline{\frac{\partial u}{\partial x_i}}\in R,
\]

\noindent where the bar denotes the image under
\(\mathbb Z[F_n]\to\mathbb Z[M_n^{\mathrm{ab}}]=R\), and the expression on the right is
computed using any lift of \(u\) to \(F_n\).  These derivatives define the
map

\[
\begin{aligned}
D\colon M_n&\longrightarrow R^{1\times n},\\
u&\longmapsto
\left(\partial_1u,\ldots,\partial_nu\right).
\end{aligned}
\]

A row over \(R\) is \emph{unimodular} if its entries generate the unit ideal.

For \(\varphi\in\operatorname{End}(M_n)\), let

\[
J(\varphi)=\left(\partial_j\varphi(x_i)\right)_{i,j}
\]
be its \emph{metabelian Fox Jacobian}.  We use \(\varphi^{\mathrm{ab}}\) for
the endomorphisms induced on \(M_n^{\mathrm{ab}}\) and on~\(R\).

We record standard Fox-calculus facts needed in the proof.

\begin{proposition}
\label{prop:fox-bachmuth}
For \(u\in M_n\) and \(\varphi\in\operatorname{End}(M_n)\), the following hold.
\begin{enumerate}[label=\textup{(\roman*)}]
  \item \(D(u)b=\overline{u}-1\), where \(\overline{u}\) is the image of \(u\)
  in \(M_n^{\mathrm{ab}}\).
  \item If \(u\) is primitive, then \(D(u)\) is unimodular over \(R\).
  \item The endomorphism \(\varphi\) is an automorphism if and only if
  \(J(\varphi)\) is invertible over \(R\).
\end{enumerate}
Moreover,

\[
D(\varphi(u))=\varphi^{\mathrm{ab}}(D(u))J(\varphi).
\]
\end{proposition}

\begin{proof}
The first identity and the chain rule are the standard Fox identities
\cite{Fox1953}.  If \(u\) is primitive, it is the first member of a basis, and
its Fox row is therefore a row of an invertible Jacobian.  Item~(iii) is
Bachmuth's inverse-function theorem \cite{Bachmuth1965}.
\end{proof}

An endomorphism is an \emph{IA-endomorphism} if it induces the identity on
\(M_n^{\mathrm{ab}}\).  For an IA-endomorphism \(\varphi\), the preceding
identities give

\begin{equation}
\label{eq:ia-fox}
D(\varphi(u))=D(u)J(\varphi),\qquad J(\varphi)b=b.
\end{equation}

An integer vector is \emph{primitive} if its coordinates are relatively prime.
We first control the induced map on abelianization.

\begin{lemma}
\label{lem:integer-primitivity}
Let \(B\in \operatorname{Mat}_{n\times n}(\mathbb Z)\).  Suppose that \(Bv\) is a primitive integer vector
whenever \(v\in\mathbb Z^n\) is primitive.  Then
\(B\in\operatorname{GL}_n(\mathbb Z)\).
\end{lemma}

\begin{proof}
For \(n=1\), the hypothesis gives \(B=\pm1\).  Let \(n\geq2\).  If
\(\det B=0\), a primitive
integer vector in the rational kernel of \(B\) is mapped to zero.  If a prime
\(p\) divides \(\det B\), choose
\(0\neq\overline v\in\ker(B\bmod p)\).  The vector \(\overline v\) has a
primitive integral lift \(v\): complete it to a matrix in
\(\operatorname{SL}_n(\mathbb F_p)\), factor this matrix into elementary
matrices, and lift the factors to \(\operatorname{SL}_n(\mathbb Z)\).  Every
entry of \(Bv\) is divisible by \(p\), so \(Bv\) is not primitive.  Both cases
contradict the hypothesis.  Hence \(\det B=\pm1\).
\end{proof}

We next isolate the specialization argument.  A primitive element is called
\emph{tame} if it is the image of a basis element under an automorphism of
\(M_n\) induced by an automorphism of \(F_n\).

Let \(\mathfrak m\) be a maximal ideal of \(R\), set
\(\Bbbk=R/\mathfrak m\), and write
\(\lambda_i=t_i+\mathfrak m\in \Bbbk^\times\).
The field \(\Bbbk\) is finite, since it is a field finitely generated as a ring.
Define

\[
b_\chi=(\lambda_1-1,\ldots,\lambda_n-1)^{\mathsf T},\qquad
H_\chi=\{q\in \Bbbk^{1\times n}:qb_\chi=0\},
\]
where the \emph{character}
\(\chi\colon M_n^{\mathrm{ab}}\to \Bbbk^\times\) is given by
\(t_i\mapsto\lambda_i\).  For \(u\in M_n\), write
\(\delta_x(u)=D(u)\bmod\mathfrak m\).

\begin{lemma}
\label{lem:finite-character}
Assume that \(n\geq2\) and that \(\chi\) is nontrivial.  If
\(0\neq K\leq H_\chi\) is a \(\Bbbk\)-subspace, then there exists a tame primitive
element \(u\in M_n\) such that

\[
0\neq\delta_x(u)\in K.
\]
\end{lemma}

\begin{proof}
The assertion is invariant under a Nielsen change of basis.  Indeed, if
\(y_1,\ldots,y_n\) is obtained from the fixed basis by Nielsen
transformations and \(C\) is the matrix whose \(i\)-th row is
\(\delta_x(y_i)\), define \(\delta_y\), \(b_{\chi,y}\), and
\(H_{\chi,y}\) relative to this basis.  The specialized chain rule gives

\[
C\in\operatorname{GL}_n(\Bbbk),\qquad
\delta_x(u)=\delta_y(u)C,\qquad
b_{\chi,y}=Cb_\chi.
\]

Thus \(H_{\chi,y}C=H_\chi\), and \(K\) is replaced by \(KC^{-1}\).

Let \(C_0\leq \Bbbk^\times\) be generated by
\(\lambda_1,\ldots,\lambda_n\), and choose a generator \(g\) of \(C_0\).
Write \(\lambda_i=g^{a_i}\).  Smith reduction of the integer row
\((a_1,\ldots,a_n)\), realized by Nielsen transformations, gives a basis for
which

\[
\chi(y_1)=g,\qquad \chi(y_i)=1,\qquad 2\leq i\leq n,
\]

because \(\gcd(a_1,\ldots,a_n)\) is coprime to \(|C_0|\); replace \(g\) by
the resulting generator of \(C_0\).  If
\(p=\operatorname{char}\Bbbk\), then \(\Bbbk=\mathbb F_p(g)\), because the images of the
Laurent variables generate \(\Bbbk\) as a ring.  Let \(e_i\) denote the \(i\)-th
standard row vector.  In these coordinates,

\[
H_\chi=\{0\}\oplus \Bbbk^{n-1}.
\]

Assume first that \(n\geq3\).  For distinct \(i,j\in\{2,\ldots,n\}\) and
\(s\in\mathbb Z\), consider the Nielsen automorphisms

\[
\nu_{ij}\colon y_i\mapsto y_i y_j,
\qquad
\rho_{i,s}\colon y_i\mapsto y_1^s y_i y_1^{-s},
\]
fixing all generators not displayed.  They preserve \(\chi\).  Write
\(J_\chi(\eta)\) for the Jacobian of \(\eta\) in the \(y\)-basis, specialized
by \(\chi\).  On \(H_\chi\), the two Jacobians restrict to

\[
I+E_{ij}
\qquad\text{and}\qquad
\operatorname{diag}(1,\ldots,g^s,\ldots,1),
\]
respectively.  Conjugating the first matrix by the second gives
\(I+g^sE_{ij}\).  Since the powers of \(g\)
span \(\Bbbk=\mathbb F_p(g)\) over \(\mathbb F_p\), products of these elementary matrices
give \(I+aE_{ij}\) for every \(a\in \Bbbk\).  Hence the restricted Jacobians contain
\(\operatorname{SL}_{n-1}(\Bbbk)\), which is transitive on the nonzero row vectors
of \(\Bbbk^{n-1}\).

Choose \(0\neq q\in K\).  A product \(\sigma\) of the preceding Nielsen
automorphisms can be chosen so that
\(e_2J_\chi(\sigma)=q\).  Then \(u=\sigma(y_2)\) is tame primitive and
\(\delta_y(u)=q\).  Transforming back to the original basis proves the claim.

If \(n=2\), then \(H_\chi=\Bbbk e_2\), so \(K=H_\chi\), and \(u=y_2\) suffices.
\end{proof}

\begin{proof}[Proof of Answer~\ref{ans:problem-14-85}]
For \(n=1\), a primitive-preserving endomorphism sends a generator to a
generator and is therefore an automorphism.  Assume \(n\geq2\).  Every
primitive integer vector belongs to a basis of \(\mathbb Z^n\), and hence has
a tame primitive lift to
\(M_n\).  Conversely, the abelianization of a primitive element is primitive.
By Lemma~\ref{lem:integer-primitivity},
the map induced by \(\varphi\) on \(M_n^{\mathrm{ab}}\cong\mathbb Z^n\) lies
in \(\operatorname{GL}_n(\mathbb Z)\).

Choose an automorphism \(\tau\) of \(M_n\), induced by an automorphism of
\(F_n\), whose action on abelianization is inverse to that of \(\varphi\), and
set \(\psi=\tau\circ\varphi\).  Then \(\psi\) is an IA-endomorphism and
primitive-preserving.  Put
\(J=J(\psi)\).  By \eqref{eq:ia-fox}, \(Jb=b\).

Suppose that \(J\) is not invertible.  Then \(\det J\) belongs to a maximal
ideal \(\mathfrak m\) of \(R\).  Let \(J_{\mathfrak m}\) be the reduction of
\(J\) over \(\Bbbk=R/\mathfrak m\), and put

\[
0\neq K=\ker_{\mathrm{left}}J_{\mathfrak m}.
\]

The corresponding character
\(\chi\colon M_n^{\mathrm{ab}}\to \Bbbk^\times\) is nontrivial.
Indeed, if every \(\lambda_i\) were \(1\), specialization would be
augmentation, and the IA condition would give \(J_{\mathfrak m}=I_n\),
contrary to its singularity.  Moreover, every \(q\in K\) satisfies

\[
qb_\chi=qJ_{\mathfrak m}b_\chi=0.
\]

Thus \(0\neq K\leq H_\chi\).  By Lemma~\ref{lem:finite-character}, there is a tame
primitive \(u\in M_n\) such that
\(0\neq\delta_x(u)\in K\).  The IA chain rule gives

\[
D(\psi(u))\bmod\mathfrak m
=\delta_x(u)J_{\mathfrak m}=0.
\]

Hence all entries of \(D(\psi(u))\) lie in \(\mathfrak m\), so this row is not
unimodular.  By Proposition~\ref{prop:fox-bachmuth}(ii), \(\psi(u)\) is not primitive,
contrary to the hypothesis.  Therefore \(J\) is invertible.
Proposition~\ref{prop:fox-bachmuth}(iii) implies that \(\psi\) is an automorphism, and so
is \(\varphi=\tau^{-1}\circ\psi\).
\end{proof}

\begin{aipractices}

We asked the model (ChatGPT 5.6 Sol) for both a proof and a counterexample,
verified the exact status and hypotheses in the literature, and only then
requested a self-contained proof.  The known theorem for ordinary free
  groups supplied a useful structural analogue.  The main step was an
algebraic reformulation: primitivity became a condition on Fox derivatives,
while failure of invertibility became a finite-field obstruction.  More
generally, useful prompts should ask for a concrete obstruction, a reduction to
a simpler algebraic setting, and the smallest class of test objects needed for
a contradiction.
\end{aipractices}

\section{Equal distributions in free metabelian groups}
\label{sec:problem-19-94}

For \(w\in M_n\), a finite metabelian group \(G\), and \(g\in G\), set

\[
N_{w,G}(g)=\bigl|\{(a_1,\ldots,a_n)\in G^n:
                 w(a_1,\ldots,a_n)=g\}\bigr|.
\]

Elements \(u,v\in M_n\) have the same \emph{distribution} if
\(N_{u,G}=N_{v,G}\) for every finite metabelian group \(G\).  Timoshenko
proved that \(w\in M_n\) is primitive if and only if
\(N_{w,G}(g)=|G|^{n-1}\) for every finite metabelian group \(G\) and every
\(g\in G\) \cite{Timoshenko2013}.  Hence the next problem has an affirmative
answer when one of the two elements is primitive.

\begin{Question}[{\cite[Problem~19.94]{kourovka21}}]
Do two elements of \(M_n\) have the same distribution if and only if an
automorphism of \(M_n\) maps one to the other?
\end{Question}

\begin{Answer}
\label{ans:problem-19-94}
No.  Let \(M_2=\langle x,y\rangle\), put \(c=[x,y]\), and use the convention
\(a^z=z^{-1}az\).  The elements

\[
\begin{aligned}
g_1&=c^{x^2y^7}(c^x)^{-1}(c^{xy^7})^{-1}c,\\
g_2&=c^{x^3y^7}(c^x)^{-1}(c^{x^2y^7})^{-1}c
\end{aligned}
\]
have the same distribution on every finite metabelian group, but belong to
distinct \(\operatorname{Aut}(M_2)\)-orbits.
\end{Answer}

\subsection{Solution}

The \emph{profinite completion} \(\widehat M_n\) of \(M_n\) is the inverse
limit of its finite quotients.

\begin{proposition}
\label{prop:profinite-orbit}
Elements \(u,v\in M_n\) have the same distribution on all finite metabelian
groups if and only if some automorphism of \(\widehat M_n\) maps \(u\) to
\(v\).
\end{proposition}

\begin{proof}
An automorphism of \(\widehat M_n\) acts by precomposition on continuous
homomorphisms to each finite group, proving one implication.  Conversely, for
a finite metabelian group \(G\) and \(z\in G\),

\[
N_{u,G}(z)=
\sum_{\substack{H\leq G\\ z\in H}}
\bigl|\{\theta\colon M_n\twoheadrightarrow H:\theta(u)=z\}\bigr|.
\]

The term with \(H=G\) is the number of epimorphisms onto \(G\), while all
other terms involve proper subgroups.  Induction on \(|G|\) therefore shows
that equality of distributions implies equality of the corresponding numbers
of epimorphisms.  Let
\(q\colon\widehat M_n\twoheadrightarrow Q\) be a finite quotient.  Applying
this observation to \(q|_{M_n}\) gives an epimorphism
\(\theta\colon M_n\twoheadrightarrow Q\) with \(\theta(v)=q(u)\).  Lift each
\(\theta(x_i)\) to \(b_i\in M_n\).  By freeness, \(x_i\mapsto b_i\) extends to
an endomorphism of \(M_n\), and hence to an endomorphism \(\alpha_q\) of
\(\widehat M_n\), with

\[
q\alpha_q(v)=q(u),\qquad q\alpha_q(\widehat M_n)=Q.
\]

Passing to a common refinement shows that the conditions obtained from any
finite collection of quotients have a common solution.  They are closed
conditions on the compact space of continuous endomorphisms.  Hence some
endomorphism \(\alpha\) satisfies them for every finite quotient.  It follows
that \(\alpha(v)=u\) and that \(\alpha(\widehat M_n)\) is dense.
It is also compact, hence closed, so \(\alpha\) is surjective.  A surjective
endomorphism of a finitely generated profinite group is injective; therefore
\(\alpha\) is an automorphism, and \(\alpha^{-1}(u)=v\).
\end{proof}

We use the following rank-two profinite Magnus description.  Let \(X,Y\) be
the images of
\(x,y\) in \(M_2/M_2'\), and put

\[
R=\mathbb Z[X^{\pm1},Y^{\pm1}],\qquad
\Lambda=\varprojlim_{m,N}
(\mathbb Z/m\mathbb Z)[(\mathbb Z/N\mathbb Z)^2],
\]
Thus \(\Lambda\) is the completed group ring of
\(\widehat{\mathbb Z}^{\,2}\), where \(\widehat{\mathbb Z}\) is the
profinite completion of \(\mathbb Z\), and
\(\varepsilon\colon\Lambda\to\widehat{\mathbb Z}\) is the
\emph{augmentation map}.  Its kernel is the \emph{augmentation ideal}.
The group \(M_2'\) is the free cyclic \(R\)-module generated by \(c\); its
closure in \(\widehat M_2\) is the free rank-one topological
\(\Lambda\)-module generated by \(c\)
\cite[Sections~1--2]{BenEzra2016}.  We write its elements as \(c^r\),
\(r\in\Lambda\).  The natural map
\(\operatorname{Aut}(\widehat M_2)\to
\operatorname{GL}_2(\widehat{\mathbb Z})\) is surjective
\cite[Sections~1--2]{BenEzra2016}.

\begin{lemma}
\label{lem:rank-two-magnus}
Let \(\rho\in\operatorname{Aut}(\widehat M_2)\) induce
\(A\in\operatorname{GL}_2(\widehat{\mathbb Z})\) on abelianization, and let
\(\sigma_A\) be the induced automorphism of \(\Lambda\).  Then

\[
\rho(c^r)=c^{d\sigma_A(r)},\qquad r\in\Lambda
\]
for some \(d\in\Lambda^\times\) satisfying
\(\varepsilon(d)=\det A\).  Conversely, if
\(s\in\Lambda^\times\) and \(\varepsilon(s)=1\), there is an automorphism
\(\tau_s\) of \(\widehat M_2\), inducing the identity on abelianization, such that
\(\tau_s(c^r)=c^{sr}\) for every \(r\in\Lambda\).
\end{lemma}

\begin{proof}
The inverse-limit Magnus embedding and the Fox chain rule give the first
statement: in rank two, the determinant of the Jacobian is the multiplier of
the basic commutator, and its augmentation is the determinant on
abelianization.  For the converse, the augmentation ideal of \(\Lambda\) is
the closed ideal generated by \(X-1\) and \(Y-1\).  Thus

\[
s=1+(Y-1)P+(1-X)Q
\]
for some \(P,Q\in\Lambda\).  The continuous endomorphism

\[
x\longmapsto xc^P,\qquad y\longmapsto yc^Q
\]
induces the identity on abelianization and sends \(c\) to \(c^s\).  Since
\(s\) is a unit, it is surjective on the closed derived group and hence on
\(\widehat M_2\).  It is an automorphism because every surjective
endomorphism of a finitely generated profinite group is injective.
\end{proof}

An element \(X^aY^b\in\Lambda^\times\), with
\(a,b\in\widehat{\mathbb Z}\), is called \emph{group-like}.

\begin{lemma}
\label{lem:profinite-power}
If \(Z\in\Lambda^\times\) is group-like and
\(a\in\widehat{\mathbb Z}^{\times}\), there is a
unit \(D_a(Z)\in\Lambda^\times\) such that

\[
Z^a-1=(Z-1)D_a(Z),\qquad \varepsilon(D_a(Z))=a.
\]
\end{lemma}

\begin{proof}
For an integer \(m\), define

\[
D_m(Z)=
\begin{cases}
\displaystyle\sum_{j=0}^{m-1}Z^j,&m>0,\\
0,&m=0,\\
-\displaystyle\sum_{j=m}^{-1}Z^j,&m<0.
\end{cases}
\]

These elements extend continuously, through the finite group rings defining
\(\Lambda\), to \(D_a(Z)\) for \(a\in\widehat{\mathbb Z}\).  The integral
identities extend by continuity and give

\[
Z^a-1=(Z-1)D_a(Z),\qquad \varepsilon(D_a(Z))=a,
\]

and

\[
D_{ab}(Z)=D_a(Z)D_b(Z^a).
\]

Taking \(b=a^{-1}\) shows that
\(D_a(Z)^{-1}=D_{a^{-1}}(Z^a)\).
\end{proof}

Set

\[
f=(X-1)(XY^7-1),\qquad h=(X-1)(X^2Y^7-1).
\]

Then the elements in the answer are \(g_1=c^f\) and \(g_2=c^h\).

\begin{lemma}
\label{lem:same-distribution}
The elements \(g_1\) and \(g_2\) lie in the same
\(\operatorname{Aut}(\widehat M_2)\)-orbit.
\end{lemma}

\begin{proof}
Choose \(u\in\widehat{\mathbb Z}^{\times}\) whose \(7\)-adic component is
\(3\) and whose other \(p\)-adic components are \(1\).  Since
\(u^{-1}\equiv5\pmod 7\),

\[
\frac{2u^{-1}-u}{7}\in\widehat{\mathbb Z}.
\]

The matrix

\[
A=\begin{pmatrix}
u&(2u^{-1}-u)/7\\
0&u^{-1}
\end{pmatrix}
\in\operatorname{SL}_2(\widehat{\mathbb Z})
\]
is induced by an automorphism \(\rho\) of \(\widehat M_2\), by the surjectivity
recorded above.  For the induced automorphism \(\sigma_A\) of \(\Lambda\),

\[
\sigma_A(X)=X^u,\qquad
\sigma_A(XY^7)=(X^2Y^7)^{u^{-1}}.
\]

Consequently, Lemma~\ref{lem:profinite-power} gives

\[
\sigma_A(f)=hD_u(X)D_{u^{-1}}(X^2Y^7).
\]

By Lemma~\ref{lem:rank-two-magnus}, \(\rho(c)=c^d\) for a unit \(d\) of
augmentation \(1\).  Hence

\[
\rho(g_1)=c^{hs},\qquad
s=dD_u(X)D_{u^{-1}}(X^2Y^7),\qquad \varepsilon(s)=1.
\]

Apply the second part of Lemma~\ref{lem:rank-two-magnus} to \(s^{-1}\).  The
resulting automorphism \(\tau\) satisfies
\((\tau\rho)(g_1)=g_2\).
\end{proof}

\begin{lemma}
\label{lem:distinct-discrete-orbits}
No automorphism of \(M_2\) maps \(g_1\) to \(g_2\).
\end{lemma}

\begin{proof}
Suppose that \(\alpha\in\operatorname{Aut}(M_2)\) induces
\(B\in\operatorname{GL}_2(\mathbb Z)\) on abelianization.  The rank-two
Magnus formula over \(R\) gives

\[
\alpha(c^r)=c^{q\sigma_B(r)}
\]
for some \(q\in R^\times\).  For a nonzero Laurent polynomial, its
\emph{Newton polygon} is the convex hull of the exponent vectors of its
nonzero monomials.  Since
\(R^\times=\{\pm X^iY^j:i,j\in\mathbb Z\}\), multiplication by \(q\) only
translates this polygon.  Up to sign, the two edge-direction vectors of the Newton
polygons of \(f\) and \(h\) are respectively

\[
\{(1,0),(1,7)\},\qquad \{(1,0),(2,7)\}.
\]

Thus \(\alpha(g_1)=g_2\) would force \(B\), up to independent signs, to map
the first pair to the second, possibly after interchanging its members.  With
no interchange, the first column of \(B\) is \((\delta,0)^{\mathsf T}\), and
the first entry of its second column is

\[
\frac{2\eta-\delta}{7},\qquad \delta,\eta\in\{\pm1\},
\]
which is not an integer.  After interchange, the corresponding entry is
\((\eta-2\delta)/7\), again not an integer.  This contradicts
\(B\in\operatorname{GL}_2(\mathbb Z)\).
\end{proof}

\begin{proof}[Proof of Answer~\ref{ans:problem-19-94}]
Lemma~\ref{lem:same-distribution} and
Proposition~\ref{prop:profinite-orbit} show that \(g_1\) and \(g_2\) have the
same distribution.  Lemma~\ref{lem:distinct-discrete-orbits} shows that their
\(\operatorname{Aut}(M_2)\)-orbits are distinct.
\end{proof}

\begin{aipractices}

We supplied the model (ChatGPT 5.6 Sol) with the solution of the preceding problem and asked it
to search for an open problem that could reasonably be attacked using the
methods developed while solving Kourovka Problem~14.85 and presented in
Section~\ref{sec:problem-14-85}.  After surveying many possible candidates, the
model selected the present problem, reformulated it in different terms, and
constructed a counterexample.  A useful strategy for researchers is therefore
to ask the model what related problems it might be able to solve by adapting a
successful method and modifying the constructions used in its proof.
\end{aipractices}

\section{Derived subgroups as Frattini subgroups}
\label{sec:problem-16-11}

For a finite \(p\)-group \(H\), its \emph{Frattini subgroup} \(\Phi(H)\) is
the intersection of all maximal subgroups of \(H\).

\begin{Question}[{\cite[Problem~16.11]{kourovka21}}]
Let \(G\) be a finite \(p\)-group.  Does there always exist a finite
\(p\)-group \(H\) such that \(\Phi(H)\cong [G,G]\)?
\end{Question}

\begin{Answer}
\label{ans:problem-16-11}
No.  For \(p=2\), one may take
\(G=\texttt{SmallGroup(256,511)}\).
\end{Answer}

The obstruction used below is due to van der Waall and de Nijs
\cite[Theorem~1.15]{vanderWaallDeNijs1995}.  It remains to identify their
exceptional group \(32/40\) as the derived subgroup of
\(\texttt{SmallGroup(256,511)}\).

\subsection{Solution}

\begin{theorem}[{van der Waall--de Nijs
\cite[Theorem~1.15]{vanderWaallDeNijs1995}}]
\label{thm:hall-senior-32-40}
The Hall--Senior group \(32/40\) is not isomorphic to \(\Phi(H)\) for any
finite \(2\)-group \(H\).
\end{theorem}

\begin{lemma}
\label{lem:derived-256-511}
If \(G=\texttt{SmallGroup(256,511)}\), then its derived subgroup
\(D=G'\) satisfies\footnote{The output
\(\texttt{IdGroup}(D)=[32,32]\) means that \(D\) is the thirty-second
isomorphism type of order \(32\) in the GAP Small Groups Library
\cite{SmallGrp}.  The output \(\texttt{IdTWGroup}(D)=[32,40]\) means that
\(D\) is group \(32/40\) in the Thomas--Wood catalogue
\cite{SONATA,ThomasWood1980}; this is the Hall--Senior numbering used in
\cite[p.~520]{vanderWaallDeNijs1995}.}

\[
|D|=32,\qquad
\texttt{IdGroup}(D)=[32,32],\qquad
\texttt{IdTWGroup}(D)=[32,40].
\]
Consequently, \(G'\) is the Hall--Senior group \(32/40\).
\end{lemma}

\begin{proof}
The following transcript records software and package versions and
performs both catalogue identifications \cite{GAP4,SmallGrp,SONATA}.

\begin{Verbatim}[commandchars=\\\{\}]
gap> GAPInfo.Version;
\textcolor{gray}{"4.15.1"}
gap> LoadPackage("smallgrp");;
gap> LoadPackage("sonata");;
gap> PackageInfo("smallgrp")[1].Version;
\textcolor{gray}{"1.5.4"}
gap> PackageInfo("sonata")[1].Version;
\textcolor{gray}{"2.9.7"}
gap> G := SmallGroup(256,511);;
gap> D := DerivedSubgroup(G);;
gap> Size(D);
\textcolor{gray}{32}
gap> IdGroup(D);
\textcolor{gray}{[ 32, 32 ]}
gap> IdTWGroup(D);
\textcolor{gray}{[ 32, 40 ]}
\end{Verbatim}

\end{proof}

\begin{proof}[Proof of Answer~\ref{ans:problem-16-11}]
By Lemma~\ref{lem:derived-256-511}, the derived subgroup of the indicated
finite \(2\)-group is the Hall--Senior group \(32/40\).  By
Theorem~\ref{thm:hall-senior-32-40}, it is not isomorphic to the Frattini
subgroup of any finite \(2\)-group.
\end{proof}

\begin{aipractices}
\label{subsec:problem-16-11-ai-practices}

We simply copied the formulation of the problem and asked the model
(ChatGPT 5.6 Sol) to solve
it.  Sometimes the stupidest method works.
\end{aipractices}

\section{Nilpotent Camina groups of unbounded class}
\label{sec:problem-17-47}

A group \(G\) is said to be a \emph{Camina group} if \(G'\neq G\) and
\(x^G=xG'\) for every \(x\notin G'\), where \(x^G\) denotes the conjugacy
class of \(x\).  Dark and Scoppola proved that every finite nilpotent Camina
group has nilpotency class at most three \cite{DarkScoppola1996}.  Muktibodh
and Ghate asked whether any bound remains valid without the finiteness
assumption \cite[p.~256]{MuktibodhGhate2013}.

\begin{Question}[{\cite[Problem~17.47]{kourovka21}}]
Let \(G\) be a nilpotent group such that \(x^G=xG'\) for every
\(x\in G\setminus G'\).  Is the nilpotency class of \(G\) bounded?
\end{Question}

\begin{Answer}
\label{ans:problem-17-47}
No.  For every integer \(c\geq2\), there exists a nilpotent Camina group of
class exactly \(c\).
\end{Answer}

\subsection{Solution}

We use \([g,h]=g^{-1}h^{-1}gh\) for group commutators.

A \emph{differential field} is a field equipped with a \emph{derivation}, that is, an
additive map \(\partial\) satisfying
\(\partial(ab)=a\partial(b)+b\partial(a)\).  We first construct one with a
surjective derivation.

\begin{lemma}
\label{lem:camina-differential-field}
There is a field \(K\) of characteristic zero with a surjective derivation
\(\partial\colon K\to K\).  For such a field, the map

\[
K\longrightarrow K,\qquad
b\longmapsto \beta(a,b):=a\partial(b)-b\partial(a)
\]
is surjective for every \(a\in K^\times\).
\end{lemma}

\begin{proof}
Start with \(K_0=\mathbb Q(t)\) and \(\partial_0=d/dt\).  Having constructed
\((K_n,\partial_n)\), adjoin algebraically independent elements \(X_f\), one
for each \(f\in K_n\), and extend the derivation to

\[
K_{n+1}=K_n(X_f:f\in K_n),
\qquad \partial_{n+1}(X_f)=f.
\]

Then \(K=\bigcup_{n\geq0}K_n\) carries a surjective derivation \(\partial\).
If \(a\neq0\) and \(z\in K\), choose \(u\in K\) with
\(\partial(u)=z/a^2\).  For \(b=au\), one has \(\beta(a,b)=z\).
\end{proof}

Fix \(c\geq2\).  Regard

\[
L_c=Ke_1\oplus\cdots\oplus Ke_c
\]
as a vector space over \(\mathbb Q\).  For notational convenience, set
\(e_i=0\) for \(i>c\).  Define a \(\mathbb Q\)-bilinear skew-symmetric bracket by

\[
\begin{aligned}
[ae_1,be_1]&=\beta(a,b)e_2,\\
[ae_1,be_i]&=ab\,e_{i+1}, && i\geq2,\\
[ae_i,be_j]&=0, && i,j\geq2.
\end{aligned}
\]
Put

\[
F^rL_c=\bigoplus_{i=r}^c Ke_i,
\qquad F^{c+1}L_c=0.
\]

Give \(Ke_i\) degree \(i\); the subspaces \(F^rL_c\) form the associated
descending filtration.  Write \(L_c'=[L_c,L_c]\) for the \emph{derived Lie
algebra}.

\begin{lemma}
The bracket above makes \(L_c\) a metabelian nilpotent Lie algebra over
\(\mathbb Q\) of class \(c\), and

\[
L_c'=F^2L_c.
\]
\end{lemma}

\begin{proof}
Only two types of Jacobi identities are nontrivial.  For
\(a,b,d,u\in K\) and \(i\geq2\), they reduce respectively to

\[
a\beta(b,d)+b\beta(d,a)+d\beta(a,b)=0
\]
and

\[
[ae_1,[be_1,ue_i]]+[be_1,[ue_i,ae_1]]=0.
\]

The first equality follows by expanding \(\beta\).  In the second Jacobi
identity, the omitted term is zero and the displayed terms are
\(abu e_{i+2}\) and \(-abu e_{i+2}\).  Lemma~\ref{lem:camina-differential-field} gives
\([Ke_1,Ke_1]=Ke_2\), while \([Ke_1,Ke_i]=Ke_{i+1}\) for
\(2\leq i<c\).  Hence \(L_c'=F^2L_c\), which is abelian, and
\([F^rL_c,L_c]=F^{r+1}L_c\) for \(2\leq r\leq c\).  Thus \(L_c\) is metabelian and
has nilpotency class \(c\).
\end{proof}

Let \(G_c=\exp(L_c)\) be the \emph{Baker--Campbell--Hausdorff group} of
\(L_c\): its elements are symbols \(\exp(x)\), with \(x\in L_c\), and its
multiplication is

\[
\exp(x)\exp(y)=\exp(\operatorname{BCH}(x,y)).
\]

The BCH series is finite because \(L_c\) is nilpotent.  For a group or Lie
algebra \(N\), its \emph{lower central series} is defined by
\(\gamma_1(N)=N\) and
\(\gamma_{r+1}(N)=[\gamma_r(N),N]\).

\begin{lemma}
\label{lem:camina-bch-group}
The group \(G_c\) has nilpotency class \(c\), and
\(G_c'=\exp(F^2L_c)\).
\end{lemma}

\begin{proof}
Under the BCH--Mal'cev correspondence, lower central series
correspond: for every nilpotent Lie algebra \(N\) over \(\mathbb Q\), not
necessarily finite-dimensional,
\[
\gamma_r(\exp N)=\exp(\gamma_rN),\qquad r\geq1;
\]
see, for example,
\cite[Theorem~5.A.2 and Section~5.B]{CornulierTessera2017}.
Taking \(r=2\) gives \(G_c'=\exp(F^2L_c)\), and the last nonzero term of the
lower central series occurs at \(r=c\).
\end{proof}

\begin{lemma}
\label{lem:camina-conjugacy-orbit}
If \(x\in L_c\setminus F^2L_c\), then

\[
\{e^{-\operatorname{ad}y}x:y\in L_c\}=x+F^2L_c.
\]
Here the exponential is finite because \(\operatorname{ad}y\) is nilpotent.
\end{lemma}

\begin{proof}
The inclusion from left to right follows from
\(e^{-\operatorname{ad}y}x-x\in L_c'=F^2L_c\).  Write

\[
x=ae_1+x_2e_2+\cdots+x_ce_c,
\qquad a\in K^\times,
\]
and take \(y=b_1e_1+\cdots+b_{c-1}e_{c-1}\).  The degree-two component of
\(e^{-\operatorname{ad}y}x-x\) is \(\beta(a,b_1)e_2\), and hence can be
prescribed arbitrarily by Lemma~\ref{lem:camina-differential-field}.  For
\(3\leq r\leq c\), after \(b_1,\ldots,b_{r-2}\) have been chosen, the
degree-\(r\) component depends affinely on \(b_{r-1}\) with coefficient
\(a\), and is independent of \(b_r,\ldots,b_{c-1}\).  Indeed, the only
degree-\(r\) occurrence of \(b_{r-1}\) in the expansion of
\(e^{-\operatorname{ad}y}x\) is
\(-[b_{r-1}e_{r-1},ae_1]=ab_{r-1}e_r\).  Since \(a\neq0\), the coefficients
\(b_{r-1}\) may be chosen successively to prescribe every homogeneous
component.  This proves the reverse inclusion.
\end{proof}

\begin{proof}[Proof of Answer~\ref{ans:problem-17-47}]
Conjugation in \(G_c\) satisfies

\[
\exp(-y)\exp(x)\exp(y)=\exp(e^{-\operatorname{ad}y}x).
\]

Modulo \(F^2L_c\), the BCH product is addition.  Hence the map
\(\exp(x)\mapsto x+F^2L_c\) from \(G_c\) to
\((L_c/F^2L_c,+)\) is a homomorphism, and
Lemma~\ref{lem:camina-bch-group} identifies its kernel with \(G_c'\).
Consequently

\[
\exp(x)G_c'=\{\exp(z):z\in x+F^2L_c\}.
\]

If \(\exp(x)\notin G_c'\), then \(x\notin F^2L_c\), and
Lemma~\ref{lem:camina-conjugacy-orbit} gives
\(\exp(x)^{G_c}=\exp(x)G_c'\).  Moreover, \(G_c'\neq G_c\) because
\(L_c/F^2L_c\neq0\).  Thus \(G_c\) is a Camina group of class \(c\) by
Lemma~\ref{lem:camina-bch-group}.  Since \(c\geq2\) is arbitrary, no bound
exists.
\end{proof}

\begin{aipractices}

The model (ChatGPT 5.6 Sol) solved the problem using the method presented in
Subsection~\ref{subsec:problem-16-11-ai-practices}.
\end{aipractices}

\section{Minimal T-systems and negative immersions}
\label{sec:question-2-8-14}

An element of a free group is \emph{primitive} if it belongs to a free basis,
and \emph{imprimitive} otherwise.  Let \(F_n\) be a free group.  The
\emph{primitivity rank} of \(w\in F_n\) is

\[
\pi(w)=\min\{\operatorname{rk}(H):H\leq F_n,\ w\in H
\text{ and }w\text{ is imprimitive in }H\},
\]
with \(\pi(w)=\infty\) if \(w\) is primitive.  Write \(X_w\) for the
\emph{natural presentation complex} of
\(G_w=F_n/\langle\!\langle w\rangle\!\rangle\): it has one vertex, \(n\)
oriented edges corresponding to a free basis, and one \(2\)-cell attached
along \(w\).

A combinatorial map of \(2\)-complexes is an \emph{immersion} if it is
locally injective.  For a finite connected \(2\)-complex \(Y\), put
\(F=\pi_1(Y^{(1)})\), and let \(w_1,\ldots,w_s\) be its attaching words.
Following Louder and Wilton, \(Y\) is \emph{reducible} if one of the following
holds:
\begin{enumerate}[
  label=\textup{(\alph*)},
  leftmargin=2em,
  itemsep=0pt,
  topsep=0.2em
]
  \item \(F=1\) and \(s=0\);
  \item \(F\cong\mathbb Z\), \(s\leq1\), and, if \(s=1\),
  \(F=\langle w_1\rangle\);
  \item \(F=A*B\) is a nontrivial free splitting and every \(w_i\) is
  conjugate into \(A\) or \(B\).
\end{enumerate}
A finite \(2\)-complex \(X\) has \emph{negative immersions} if every immersion
\(Y\looparrowright X\) of finite connected \(2\)-complexes satisfies
\(\chi(Y)<0\) or \(Y\) is reducible.  A group is \emph{\(2\)-free} if each
of its subgroups generated by at most two elements is free.  Finally, \(X\)
has \emph{uniform negative immersions} if there is \(\varepsilon>0\) such
that every such immersion, with \(f_2(Y)\) denoting the number of \(2\)-cells
of \(Y\), satisfies

\[
\frac{\chi(Y)}{f_2(Y)}\leq-\varepsilon
\]
or \(Y\) is reducible
\cite[Proposition~3.8 and Definitions~3.20, 3.24]{LouderWilton2024}.

The following result is due to Louder and Wilton
\cite{LouderWilton2022,LouderWilton2024}.

\begin{theorem}
\label{thm:negative-immersions-primitivity-rank}
The following are equivalent: \(X_w\) has negative immersions;
\(G_w\) is \(2\)-free; and \(\pi(w)\geq3\).  If these conditions hold, then
\(X_w\) has uniform negative immersions.
\end{theorem}

Two generating \(n\)-tuples of a group \(G\) are \emph{Nielsen equivalent}
if they differ by the action of \(\operatorname{Aut}(F_n)\).  They belong to
the same \emph{\(T\)-system} if an automorphism of \(G\) takes one to a
Nielsen-equivalent tuple.  A positive answer to the following question was
known for closed surface groups \cite[Question~2.8.14]{LintonNybergBrodda2025}.

\begin{Question}[{\cite[Question~2.8.14]{LintonNybergBrodda2025}}]
Let \(X\) be a one-relator complex with negative immersions.  Does
\(\pi_1(X)\) have a unique \(T\)-system of generating sets of minimal
cardinality?
\end{Question}

\begin{Answer}
\label{ans:question-2-8-14}
No.  For every \(N\geq1\), there is a three-generated \(2\)-free one-relator
group with at least \(N\) distinct \(T\)-systems of generating triples of
minimal cardinality.  Each triple is induced by a one-relator presentation
whose natural complex has uniform negative immersions.
\end{Answer}

\subsection{Solution}

We first establish isomorphisms among certain one-relator presentations.  In this
section, \(\langle\!\langle u\rangle\!\rangle\) denotes the \emph{normal closure} of
\(u\), the smallest normal subgroup containing \(u\), and
\(F_3=F(a,b,c)\).

\begin{lemma}
\label{lem:cyclic-defect}
Let \(\varphi,\psi\colon F_3\to F_3\) fix \(a,b\), and put

\[
R=(\psi\circ\varphi)(c)c^{-1},
\qquad
S=(\varphi\circ\psi)(c)c^{-1}.
\]

Then \(\varphi\) and \(\psi\) induce mutually inverse isomorphisms

\[
F_3/\langle\!\langle R\rangle\!\rangle
\cong F_3/\langle\!\langle S\rangle\!\rangle.
\]
\end{lemma}

\begin{proof}
If an endomorphism \(\eta\) fixes \(a,b\), then

\[
\eta(u)u^{-1}\in
\langle\!\langle\eta(c)c^{-1}\rangle\!\rangle,
\qquad u\in F_3,
\]
because \(\eta\) induces the identity on
\(F_3/\langle\!\langle\eta(c)c^{-1}\rangle\!\rangle\).
Applying this observation to \(\varphi\circ\psi\) and
\(\psi\circ\varphi\) shows that \(\varphi(R)\) lies in the normal closure
of \(S\), and \(\psi(S)\) lies in the normal closure of \(R\).  The two maps
therefore descend.  Their composites fix \(a,b\) and fix \(c\) modulo the
corresponding relator, so they are inverse.
\end{proof}

The \emph{exponent vector} of a word in \(F(a,b,c)\) is its image in
\(F(a,b,c)^{\mathrm{ab}}\cong\mathbb Z^3\), relative to the ordered basis
\((a,b,c)\).  A \emph{cyclic reduction} of a word is a cyclically reduced
conjugate.  For a cyclic word \(z\), write \(|z|_{\mathrm{cyc}}\) for its
cyclically reduced length and let

\[
\ell_{\mathrm W}(z)=
\min_{\alpha\in\operatorname{Aut}(F_3)}
|\alpha(z)|_{\mathrm{cyc}}
\]
be its \emph{Whitehead-minimal length}.

\begin{proposition}
\label{prop:cyclic-orbit-criterion}
Let \(\varphi_1,\ldots,\varphi_m\) fix \(a,b\), and suppose that every
\(\varphi_i(c)\) has exponent vector \((0,0,1)\).  For \(1\leq i\leq m\),
let \(Q_i\) be the cyclic reduction of

\[
(\varphi_{i-1}\circ\cdots\circ\varphi_1\circ
 \varphi_m\circ\cdots\circ\varphi_i)(c)c^{-1},
\]
where an empty composition is the identity.  Assume that every \(Q_i\) is
nontrivial, that one of them has primitivity rank \(3\), and that the
numbers \(\ell_{\mathrm W}(Q_i)\) are pairwise distinct.  Then the quotients
\(F_3/\langle\!\langle Q_i\rangle\!\rangle\) are isomorphic and their
standard generating triples determine at least \(m\) distinct \(T\)-systems
of minimal cardinality.  Every corresponding presentation complex has
uniform negative immersions.
\end{proposition}

\begin{proof}
Splitting a cyclic product into one factor and the remaining factors,
Lemma~\ref{lem:cyclic-defect} gives isomorphisms between adjacent cyclic
rotations.  Every cyclic composition induces the identity on abelianization.
Thus every \(Q_i\) has zero exponent vector, so each quotient has
abelianization \(\mathbb Z^3\) and its standard triple is minimal.

By Theorem~\ref{thm:negative-immersions-primitivity-rank}, the quotient
defined by a \(Q_i\) of primitivity rank \(3\) is \(2\)-free.  This is an
isomorphism invariant, so the same theorem gives \(\pi(Q_i)>2\) for every
\(i\).  A primitive word has an exponent vector with coprime coordinates,
and in particular a nonzero exponent vector.
Thus each nontrivial \(Q_i\) is imprimitive in \(F_3\), whence
\(\pi(Q_i)=3\).  The same theorem gives uniform negative immersions.

For normal subgroups \(N,N'\trianglelefteq F_3\), the standard triples of
\(F_3/N\) and \(F_3/N'\) represent the same \(T\)-system precisely when there
are \(\alpha\in\operatorname{Aut}(F_3)\) and an isomorphism
\(\beta\colon F_3/N\to F_3/N'\) such that
\(\beta\circ q=q'\circ\alpha\), where \(q,q'\) are the quotient maps.
Equivalently, \(\alpha(N)=N'\).
If the triples defined by \(Q_i\) and \(Q_j\) belonged to the same
\(T\)-system, then for some \(\alpha\in\operatorname{Aut}(F_3)\), the words
\(\alpha(Q_i)\) and \(Q_j\) would have the same normal closure.  Magnus's
normal-closure theorem implies that \(\alpha(Q_i)\) is conjugate to
\(Q_j^{\pm1}\) \cite[Section~2.8.3]{LintonNybergBrodda2025}.  This would imply
\(\ell_{\mathrm W}(Q_i)=\ell_{\mathrm W}(Q_j)\), a contradiction.
\end{proof}

We now give the family.  In displayed words, we use typewriter letters and
denote inverses by capitals.  Define endomorphisms
\(\rho,\sigma\), fixing \(\mathtt a,\mathtt b\), by

\[
\rho(\mathtt c)=\mathtt{CAcac},
\qquad
\sigma(\mathtt c)=\mathtt{BCbcc}.
\]

For \(m\geq3\) and \(0\leq j<m\), let \(R_{m,j}\) be the cyclic reduction
of

\[
(\rho^j\circ\sigma\circ\rho^{m-1-j})(\mathtt c)\mathtt C.
\]

These are the relators \(Q_i\) from
Proposition~\ref{prop:cyclic-orbit-criterion} for a cyclic list with
\(m-1\) copies of \(\rho\) and one \(\sigma\).

\begin{lemma}
\label{lem:tsystem-whitehead-calculation}
For every \(m\geq3\),

\[
\ell_{\mathrm W}(R_{m,0})
=13\cdot3^{m-2}-5
\]
and, for \(1\leq j<m\),

\[
\ell_{\mathrm W}(R_{m,j})
=5\cdot3^{m-2}+4\cdot3^{m-1-j}-1.
\]

In particular, these \(m\) lengths are pairwise distinct.
\end{lemma}

\begin{proof}
A \emph{cyclic word} \(z\) is a cyclically reduced word read up to cyclic
permutation.  Write \(z=z_1\cdots z_\ell\), with subscripts taken modulo
\(\ell\).  Its \emph{Whitehead graph} is the unoriented multigraph with vertex
set \(\{\mathtt a,\mathtt A,\mathtt b,\mathtt B,\mathtt c,\mathtt C\}\) in
which each adjacent pair \(z_i z_{i+1}\), called a \emph{cyclic turn}
(including \(z_\ell z_1\)), contributes an edge joining \(z_i^{-1}\) to
\(z_{i+1}\).  Let \(n_{xy}(z)\) be the multiplicity of the edge joining
\(x\) and \(y\).  For a vertex set \(A\), let \(|\partial A|_z\) be the total
multiplicity of the edges crossing from \(A\) to its complement, and let
\(\deg_z(x)\) be the weighted degree of \(x\).  A \emph{type-II Whitehead
automorphism} has a multiplier \(x\) and a defining set \(A\), where
\(x\in A\) and \(x^{-1}\notin A\).  It fixes \(x\) and sends each oriented
letter \(y\notin\{x,x^{-1}\}\) to \(y\), \(yx\), \(x^{-1}y\), or \(x^{-1}yx\)
according as neither, only \(y\), only \(y^{-1}\), or both of
\(y,y^{-1}\) belong to \(A\).  Its change in cyclic length is

\[
\Delta_x(A)=|\partial A|_z-\deg_z(x).
\]
This is Whitehead's length-change formula \cite{MKS1976}.

Let \(\iota\) invert each basis element.  For the initial value \(m=3\),
Whitehead descent followed by \(\iota\) gives the following cyclic
representatives (capitals still denote inverses):

\[
\begin{aligned}
W_{3,0}&=\mathtt{cabCBccACCbcBaaCCbcBAAbCBccaCCbcBA},\\
W_{3,1}&=\mathtt{caccAbCBCabcBaCCAAcbcBaCCA},\\
W_{3,2}&=\mathtt{bcBcacACaaCCAAcaCA}.
\end{aligned}
\]

Each is obtained from the corresponding \(R_{3,j}\) by at most seven
type-II transformations, so the base case is checked directly from the
definition above.

For the induction, expand one additional copy of
\(\rho(\mathtt c)=\mathtt{CAcac}\), apply the type-II formulas before free
reduction, and retain the states \(j=0\), \(1\leq j\leq m-2\), and
\(j=m-1\).  From \(m\) to \(m+1\), an internal state triples every
multiplicity except that
\(n_{\mathtt{aA}}\mapsto3n_{\mathtt{aA}}+2\).  The old terminal state gives
both the new internal state with \(r=1\) and the new terminal state; the two
cyclic boundary reductions give the constant terms below.  This constructs
\(W_{m,j}\) in the automorphic orbit of \(R_{m,j}\) and proves by induction
on \(m\) the formulas for \(n_{xy}(W_{m,j})\) displayed below.

Put \(p=3^{m-2}\), and write \(n_{xy}=n_{xy}(W_{m,j})\).  For \(j=0\), the
nonzero multiplicities are

\[
\begin{array}{@{}l@{\qquad}l@{}}
n_{\mathtt{aA}}=n_{\mathtt{ab}}=p-1,
&n_{\mathtt{ac}}=1,\\[2pt]
n_{\mathtt{aC}}=n_{\mathtt{Ab}}=n_{\mathtt{AC}}=p,
&n_{\mathtt{bc}}=n_{\mathtt{Bc}}=n_{\mathtt{BC}}
=n_{\mathtt{cC}}=2p-1.
\end{array}
\]

For the internal cases \(1\leq j\leq m-2\), put \(r=3^{m-2-j}\).  Then

\[
\begin{array}{@{}l@{\qquad}l@{\qquad}l@{}}
n_{\mathtt{aA}}=p-r-1,
&n_{\mathtt{ab}}=3r,
&n_{\mathtt{ac}}=n_{\mathtt{Ac}}=n_{\mathtt{AC}}=p,\\[2pt]
n_{\mathtt{aC}}=p-2r,
&n_{\mathtt{Ab}}=r,
&n_{\mathtt{bC}}=2r,\\[2pt]
n_{\mathtt{Bc}}=n_{\mathtt{BC}}=n_{\mathtt{cC}}=3r.
\end{array}
\]

For the terminal case \(j=m-1\),

\[
\begin{array}{@{}l@{\qquad}l@{}}
n_{\mathtt{aA}}=n_{\mathtt{ac}}=p-1,
&n_{\mathtt{aC}}=n_{\mathtt{Ac}}=n_{\mathtt{AC}}=p,\\[2pt]
\multicolumn{2}{@{}l@{}}{
n_{\mathtt{ab}}=n_{\mathtt{bc}}=n_{\mathtt{Bc}}
=n_{\mathtt{BC}}=n_{\mathtt{cC}}=1.}
\end{array}
\]

All omitted multiplicities are zero.

\emph{Type-I Whitehead automorphisms} are signed permutations of the basis and
preserve cyclic length.  Hence it suffices to test type-II automorphisms.
For a fixed multiplier, membership of the four letters other than
\(x^{\pm1}\) determines \(A\), so there are at most \(6\cdot2^4=96\) cases.
Substitution of the displayed multiplicities gives the following uniform
bounds.  In the cases \(j=0\) and \(j=m-1\), every nonzero \(\Delta_x(A)\) is
at least \(\min\{2,p-1\}\).  For \(1\leq j\leq m-2\), every nonzero
\(\Delta_x(A)\) is at least

\[
\min\{2r,p-1,3p-6r-1\}.
\]

Here \(p\geq3\), and for \(1\leq j\leq m-2\) one has
\(p\geq3r\) and \(r\geq1\).
Hence every \(\Delta_x(A)\) is nonnegative.  Whitehead's theorem shows that
each \(W_{m,j}\) is minimal.

Summing the turn multiplicities gives

\[
|W_{m,0}|=13p-5,
\qquad
|W_{m,j}|=5p+4\cdot3^{m-1-j}-1,\qquad 1\leq j\leq m-2,
\qquad
|W_{m,m-1}|=5p+3.
\]

These are the asserted formulas.  The lengths with \(j\geq1\) decrease
strictly with \(j\), and the difference between the first two lengths is
\(4(p-1)>0\).  They are therefore pairwise distinct.
\end{proof}

\begin{lemma}
\label{lem:tsystem-primitivity-rank}
For every \(m\geq3\), one has \(\pi(R_{m,m-1})=3\).
\end{lemma}

\begin{proof}
Put \(k=m-1\) and \(x_k=\rho^k(\mathtt c)\).  Since

\[
\rho(\mathtt c)=(\mathtt{ac})^{-1}\mathtt c(\mathtt{ac}),
\]
induction shows that every \(x_k\) is conjugate to \(\mathtt c\).  Thus it is
\emph{root-free}, meaning that it is not a proper power.  Let \(r_k\) be the
word obtained by freely reducing

\[
\mathtt B\,x_k^{-1}\mathtt b\,x_k^2\mathtt C.
\]

It is cyclically reduced and represents \(R_{m,m-1}\).
Direct free reduction gives

\[
|x_k|_{\mathrm{cyc}}=1,
\qquad
|x_k^2\mathtt C|_{\mathrm{cyc}}=2\cdot3^k-1;
\]
the second formula follows by induction, since the cyclic length equals
\(1\) for \(k=0\) and the next substitution triples it and adds \(2\).
Since \(k=m-1\geq2\), the two cyclic lengths are different.  Thus
\(x_k^2\mathtt C\) is not conjugate to a power of \(x_k\), since
abelianization would force the exponent of that power to be \(1\).

The cyclic word \(r_k\) contains \(\mathtt B\) exactly once, whereas every
letter count in a proper cyclic power is divisible by its exponent.  Thus
\(r_k\) is not a proper power and \(\pi(r_k)>1\).  Suppose that it is
imprimitive in a rank-two subgroup.  Let \(\Gamma\) be the \emph{folded core
graph} of that subgroup, namely the finite folded labeled based graph whose
based loops represent the elements of that subgroup.  The word \(r_k\) labels
a based loop in
\(\Gamma\), and \(\operatorname{rk}\pi_1(\Gamma)=2\).  The two complementary subpaths have
labels \(x_k^{-1}\) and \(x_k^2\mathtt C\), so neither contains
\(\mathtt b^{\pm1}\).  If the two occurrences of \(\mathtt b^{\pm1}\)
cross distinct unoriented edges, each such edge is crossed once.  An edge
crossed exactly once by a closed path cannot separate \(\Gamma\).  Choosing a
maximal tree that omits one of these edges makes a basis letter occur once in
\(r_k\), so \(r_k\) is primitive in \(\pi_1(\Gamma)\), a contradiction.
Hence both occurrences cross one edge \(e\) in opposite directions.

If \(e\) separates, the two components of \(\Gamma\setminus e\) have positive
rank because the two complementary subpaths are nontrivial.  Their ranks sum
to two, so both components are cyclic.  Root-freeness makes the
\(x_k^{-1}\)-loop a generator of its component.  Relative to generators of
the two cyclic components, \(r_k\) contains one generator exactly once and is
therefore primitive in \(\pi_1(\Gamma)\), again a contradiction.

If \(e\) does not separate, \(\Gamma\setminus e\) is connected of rank one,
so its fundamental group is cyclic; changing the basepoint only conjugates
this cyclic subgroup.  Since \(x_k^{-1}\) is root-free, this forces
\(x_k^2\mathtt C\) to be conjugate to a power of \(x_k\), contrary
to the preceding paragraph.  Thus \(\pi(r_k)>2\).  The nontrivial word
\(r_k\) has zero exponent vector and is therefore imprimitive in \(F_3\).
Hence \(\pi(R_{m,m-1})=\pi(r_k)=3\).
\end{proof}

\begin{proof}[Proof of Answer~\ref{ans:question-2-8-14}]
Both \(\rho(\mathtt c)\) and \(\sigma(\mathtt c)\) have exponent vector
\((0,0,1)\).  For every \(m\geq3\),
Lemmas~\ref{lem:tsystem-whitehead-calculation} and
\ref{lem:tsystem-primitivity-rank} verify all hypotheses of
Proposition~\ref{prop:cyclic-orbit-criterion}.  Hence the common quotient has
at least \(m\) distinct \(T\)-systems of minimal generating triples, and every
associated natural presentation complex has uniform negative immersions.
Given \(N\), take \(m\geq\max\{3,N\}\).
\end{proof}

\begin{aipractices}

We supplied the model (ChatGPT 5.6 Sol) with the preprint \cite{Semidetnov2026} and asked it to
identify other open problems to which the same construction might apply.
Exact computer searches produced small candidate one-relator presentations, after which
further prompts requested a parametrized family and an invariant distinguishing
its \(T\)-systems.  The resulting proof uses cyclic compositions of
endomorphisms to obtain isomorphic presentations and Whitehead-minimal length
to distinguish the \(T\)-systems.  An initial family required an unresolved
primitivity argument and was replaced by the power-endomorphism family above,
to which the existing \(2\)-freeness criterion applies.  In this procedure,
exact computation was used to formulate statements, but not as a substitute
for proof.
\end{aipractices}

\section{Word-defined group laws on free nilpotent groups}
\label{sec:neshchadim}

For a group \(G\), let

\[
\gamma_1(G)=G,
\qquad
\gamma_{k+1}(G)=[\gamma_k(G),G]
\]
be its \emph{lower central series}.  Let \(F_n\) be the free group of rank
\(n\).  For a word \(w=w(a,b)\in F_2\), the \emph{word-defined binary
operation} on a group \(G\) associated with \(w\) is
\(x*_w y=w(x,y)\).

\begin{Question}[{\cite[Problem~18, p.~400]{BardakovEtAl2015}}]
Let \(w=w(a,b)\in F_2\) be reduced.  Suppose that \(x*_w y=w(x,y)\)
defines a group operation on \(F_n/\gamma_k(F_n)\) for all \(n\geq2\) and
\(k\geq1\).  Must \(w=ab\) or \(w=ba\)?
\end{Question}

\begin{Answer}
\label{ans:neshchadim}
Yes.  The only such words are \(ab\) and \(ba\).
\end{Answer}

\subsection{Solution}

We use Hanna Neumann's theorem in the form stated for free groups in
\cite[p.~142]{HulanickiSwierczkowski1962}.

\begin{theorem}[{Neumann \cite[Theorem~2.1]{Neumann1961}}]
Let \(G\) be a nonabelian free group, let \(F(x,y)\) be the free group with
basis \(\{x,y\}\), and let \(f=f(x,y)\in G*F(x,y)\).  For \(u,v\in G\),
let \(f(u,v)\) denote evaluation fixing the coefficient subgroup \(G\) and
sending \(x\) to \(u\) and \(y\) to \(v\).  If
\[
u\circ v:=f(u,v)
\]
defines an associative binary operation on \(G\), then, as a function
\(G\times G\to G\), it is one of
\[
(u,v)\mapsto z,\qquad
(u,v)\mapsto u,\qquad
(u,v)\mapsto v,\qquad
(u,v)\mapsto uzv,\qquad
(u,v)\mapsto vzu
\]
for some constant \(z\in G\).
\end{theorem}

\begin{proof}[Proof of Answer~\ref{ans:neshchadim}]
Let \(\alpha\) and \(\beta\) be the exponent sums of \(a\) and \(b\)
in \(w\).  On
\(F_2/\gamma_2(F_2)\cong\mathbb Z^2\), we have
\[
u*_w v=\alpha u+\beta v.
\]
Since this is a group law, its identity axioms force
\(\alpha=\beta=1\).

Let \(a,b,c\) freely generate \(F_3\), and put
\[
q=w(w(a,b),c)\,w(a,w(b,c))^{-1}.
\]
Associativity on \(F_3/\gamma_k(F_3)\) gives
\(q\in\gamma_k(F_3)\) for every \(k\).  Free groups are residually nilpotent
\cite{MKS1976}, so
\[
\bigcap_{k\geq1}\gamma_k(F_3)=1
\]
and hence \(q=1\).  Thus \(w\) defines an associative operation on the
nonabelian free group \(F_2\).

Neumann's theorem says that this operation on \(F_2\), as a function, is one
of the five displayed types.  The exponent sums \(\alpha=\beta=1\) exclude
the constant function and the two projections.  In each remaining type,
\(w(1,1)=1\) forces the coefficient to be trivial.  Hence the function is
\((u,v)\mapsto uv\) or \((u,v)\mapsto vu\).  Evaluating at the free basis
\((a,b)\) gives \(w=ab\) or \(w=ba\).  Conversely, these words give the usual
and opposite group laws, respectively.
\end{proof}

\begin{aipractices}

We copied the formulation of the problem and asked the model (Claude Opus 5) to solve it.  The
model produced a long self-contained argument.  We then requested a literature
search, which led to a much shorter proof using a theorem of Hanna
Neumann.  The useful lesson is to follow a self-contained solution with a
targeted search for classical results: the right theorem may collapse a long
proof.
\end{aipractices}

\section{Stable representations of nilpotent groups}
\label{sec:plotkin-stable-representations}

Let \(K\) be a commutative ring and \(L\) a group.  The \emph{augmentation
map} \(\varepsilon\colon KL\to K\) sends every element of \(L\) to \(1\),
and its kernel \(\Delta_L\) is the \emph{augmentation ideal}.  A representation
\((M,L)\), where \(M\) is a right \(KL\)-module, is
\emph{\(r\)-step stable} if

\[
M\Delta_L^r=0.
\]

The class of such representations is denoted by \(\mathfrak S^r\).  Plotkin
posed the following problem in 1977.  It was repeated as
Problem~12.3.8 in the monograph of Plotkin and Vovsi
\cite[Problem~12.3.8, p.~147]{PlotkinVovsi1983}.

\begin{Question}[{\cite[Problem~4.3.5, p.~33]{Plotkin1977}}]
Let \(n\geq2\), let \(G\) be a torsion-free nilpotent group of class
\(n-1\), and let \(H\trianglelefteq G\) be abelian.  Does \(G\) admit a
faithful representation in \(\mathfrak S^n\) whose restriction to \(H\) is
\(2\)-step stable?
\end{Question}

\begin{Answer}
\label{ans:plotkin-stable-representations}
Yes.  Such a representation exists over \(\mathbb Q\).
\end{Answer}

\subsection{Solution}

We use the Mal'cev correspondence \cite{Malcev1949}.

\begin{theorem}
\label{thm:plotkin-malcev}
Every torsion-free nilpotent group \(G\) embeds in \(\exp(\mathfrak g)\), where
\(\mathfrak g\) is a nilpotent Lie algebra over \(\mathbb Q\) of the same
nilpotency class as \(G\),
\(\log G\) spans \(\mathfrak g\), and multiplication is given by the
Baker--Campbell--Hausdorff formula.
\end{theorem}

Apply Theorem~\ref{thm:plotkin-malcev} and set

\[
\mathfrak a=\operatorname{span}_{\mathbb Q}(\log H)\leq\mathfrak g.
\]

\begin{lemma}
The subspace \(\mathfrak a\) is an abelian ideal of \(\mathfrak g\).
\end{lemma}

\begin{proof}
If \(x,y\in\log H\), then \(\exp x\) and \(\exp y\) commute.  Hence
\(e^{\operatorname{ad}x}y=y\).  Since \(\operatorname{ad}x\) is nilpotent,
the polynomial
\[
P(\operatorname{ad}x)
=\frac{e^{\operatorname{ad}x}-1}{\operatorname{ad}x}
=1+\frac{\operatorname{ad}x}{2!}
+\frac{(\operatorname{ad}x)^2}{3!}+\cdots
\]
is finite and invertible.  The equality
\(0=(e^{\operatorname{ad}x}-1)y
=P(\operatorname{ad}x)[x,y]\) therefore gives \([x,y]=0\).  Thus
\(\mathfrak a\) is abelian.

For \(x\in\log G\), normality of \(H\) gives

\[
e^{\operatorname{ad}x}(\log H)=\log H.
\]

Consequently \(e^{\operatorname{ad}x}\) preserves \(\mathfrak a\).  Since
\(\log(e^{\operatorname{ad}x})=\operatorname{ad}x\) is a finite polynomial
in \(e^{\operatorname{ad}x}-1\),
\(\operatorname{ad}x\) also preserves \(\mathfrak a\).  The set \(\log G\)
spans \(\mathfrak g\), so \(\mathfrak a\) is an ideal.
\end{proof}

Let \(U=U_{\mathbb Q}(\mathfrak g)\) be the \emph{universal enveloping
algebra}, the associative algebra generated by \(\mathfrak g\) subject to
\(xy-yx=[x,y]\).  Let \(\omega\) be the kernel of the augmentation
\(U\to\mathbb Q\) that vanishes on \(\mathfrak g\), and let \(J\) be the
two-sided ideal generated by all products
\(xy\), where \(x,y\in\mathfrak a\).  Put

\[
B=U/(\omega^n+J).
\]

\begin{lemma}
\label{lem:plotkin-pbw-injection}
The canonical map \(\mathfrak g\to B\) is injective.
\end{lemma}

\begin{proof}
The \emph{lower central series} of \(\mathfrak g\) is defined by
\(\gamma_1(\mathfrak g)=\mathfrak g\) and
\(\gamma_{i+1}(\mathfrak g)=[\gamma_i(\mathfrak g),\mathfrak g]\).
Choose a basis of \(\mathfrak g\) adapted both to the lower central series
and to its intersections with \(\mathfrak a\), and order the basis elements
from \(\mathfrak a\) first.  By the Poincar\'e--Birkhoff--Witt (PBW) theorem,
the ordered monomials in this basis form a basis of \(U\); call them the
\emph{PBW monomials}.  Give a basis element \(x\) the weight
\(\max\{i:x\in\gamma_i(\mathfrak g)\}\), and a PBW monomial the sum of the
weights of its factors.  The theorem gives

\[
\omega^r=
\operatorname{span}_{\mathbb Q}
\{\text{PBW monomials of weight at least }r\}.
\tag{7.1}
\]

Indeed, \(\gamma_i(\mathfrak g)\subseteq\omega^i\), so every PBW monomial of
weight at least \(r\) belongs to \(\omega^r\).  Conversely, \(\omega^r\) is
spanned by products of at least \(r\) elements of \(\mathfrak g\), and
reordering such a product into PBW form does not decrease its weight, since
\(yx=xy+[y,x]\) and
\([\gamma_i(\mathfrak g),\gamma_j(\mathfrak g)]
 \subseteq\gamma_{i+j}(\mathfrak g)\).  Since \(\mathfrak a\) is an abelian
ideal, the PBW monomials containing at least two factors from \(\mathfrak a\)
span a two-sided ideal containing the generators of \(J\).  Conversely, the
basis elements from \(\mathfrak a\) occur first in the chosen order, so every
such monomial contains a product of two elements of \(\mathfrak a\) as a
factor and belongs to \(J\).  Hence this ideal equals \(J\).

Now \(\gamma_n(\mathfrak g)=0\).  Hence no PBW monomial with one factor
belongs to \(\omega^n\) or \(J\).  PBW linear independence gives
\(\mathfrak g\cap(\omega^n+J)=0\).
\end{proof}

\begin{proof}[Proof of Answer~\ref{ans:plotkin-stable-representations}]
For \(z\in U\), write \(\bar z\) for its image in \(B\), and let
\(\bar\omega\) be the image of \(\omega\).  For \(g\in G\), define

\[
u_g=\exp(\overline{\log g})\in B^\times.
\]

The exponential is a finite sum because \(\bar\omega^n=0\), and the BCH
formula gives \(u_{gh}=u_gu_h\).  Let \(G\) act on the additive group of
\(B\) by right multiplication by \(u_g\); this is a representation over
\(\mathbb Q\).

This action is faithful.  Indeed, if right multiplication by \(u_g\) is the
identity, evaluating it at \(1\in B\) gives \(u_g=1\).  Applying the finite
logarithm and Lemma~\ref{lem:plotkin-pbw-injection} gives \(\log g=0\), hence
\(g=1\).

For every \(g\in G\), one has \(u_g-1\in\bar\omega\).  Therefore every
product of \(n\) operators \(u_g-1\) vanishes.  Since \(\Delta_G\) is spanned
by the elements \(g-1\), this is precisely \(B\Delta_G^n=0\), so the
representation lies in \(\mathfrak S^n\).  If \(h\in H\), then
\(\log h\in\mathfrak a\) and
\((\log h)^2\in J\).  Hence

\[
u_h=1+\overline{\log h}
\quad\text{and}\quad
(u_{h_1}-1)(u_{h_2}-1)=0
\]
for all \(h_1,h_2\in H\).  Thus the restriction to \(H\) is \(2\)-step
stable because \(\Delta_H\) is spanned by \(h-1\), \(h\in H\).
\end{proof}

\begin{aipractices}

We first supplied the model (ChatGPT 5.6 Sol) with the Plotkin--Vovsi book
\cite{PlotkinVovsi1983} and its OCR transcription, asking it to reconstruct
Problem~12.3.8, resolve terminological ambiguities, and restate the problem in
modern English.  We then requested a solution, checked it against the original
hypotheses, requested separate justifications for individual steps, and asked
for a literature search.
\end{aipractices}

\section{The Cayley--Hamilton analogue for free groups}
\label{sec:problem-17-32}

Let \(F_n\) be the free group of rank \(n\).  The Cayley--Hamilton theorem
expresses \(\alpha^n(v)\) as an integral linear combination of
\(v,\alpha(v),\ldots,\alpha^{n-1}(v)\), and therefore implies the following
statement for the free abelian group
\(\mathbb Z^n\): if \(v,\alpha(v),\ldots,\alpha^n(v)\) generate
\(\mathbb Z^n\), where \(\alpha\in\operatorname{GL}(n,\mathbb Z)\), then
\(v,\alpha(v),\ldots,\alpha^{n-1}(v)\) already generate
\(\mathbb Z^n\).  Problem~17.32 asks whether the analogous assertion holds
in a free group.

\begin{Question}[{\cite[Problem~17.32]{kourovka21}}]
Let \(w\in F_n\) and \(\varphi\in\operatorname{Aut}(F_n)\).  If
\[
\langle w,\varphi(w),\ldots,\varphi^n(w)\rangle=F_n,
\]
must
\[
\langle w,\varphi(w),\ldots,\varphi^{n-1}(w)\rangle=F_n?
\]
\end{Question}

\begin{Answer}
\label{ans:problem-17-32}
No.  The assertion fails for \(n=2\).
\end{Answer}

\subsection{Solution}

\begin{proof}[Proof of Answer~\ref{ans:problem-17-32}]
Let \(F_2=\langle x,y\rangle\), set \(w=xy\), and define

\[
\varphi(x)=yx,\qquad \varphi(y)=x.
\]

This is an automorphism because
\(\varphi^{-1}(x)=y\) and \(\varphi^{-1}(y)=xy^{-1}\).  Direct calculation
gives

\[
\varphi(w)=yx^2,
\qquad
\varphi^2(w)=xyxyx=w^2x.
\]

Hence \(x=w^{-2}\varphi^2(w)\) and \(y=x^{-1}w\), so
\(\langle w,\varphi(w),\varphi^2(w)\rangle=F_2\).

Now consider the epimorphism

\[
\pi\colon F_2\longrightarrow S_3,
\qquad
\pi(x)=(1\,2\,3),
\qquad
\pi(y)=(2\,3).
\]

One has \(\pi(w)=\pi(\varphi(w))=(1\,2)\).  Therefore
\(\pi(\langle w,\varphi(w)\rangle)\) has order \(2\), whereas
\(\pi(F_2)=S_3\); hence \(\langle w,\varphi(w)\rangle\) is proper in
\(F_2\).  This gives the required counterexample for \(n=2\).
\end{proof}

\begin{aipractices}

This solution was obtained after the authors asked the model (ChatGPT 5.6 Sol) to find one more
problem in the Kourovka Notebook \cite{kourovka21}, partly to bring the total
number of solved problems to eight, and left it to work overnight.  The
counterexample was found within about two hours.
\end{aipractices}


\begin{thebibliography}{99}
\small

\bibitem[B65]{Bachmuth1965}
S.~Bachmuth,
\emph{Automorphisms of free metabelian groups},
Trans. Amer. Math. Soc. \textbf{118} (1965), 93--104.
\url{https://doi.org/10.1090/S0002-9947-1965-0180597-3}

\bibitem[BEOH24]{SmallGrp}
H.~U. Besche, B.~Eick, E.~O'Brien, and M.~Horn,
\emph{SmallGrp: The GAP Small Groups Library}, Version~1.5.4,
GAP package, 2024.
\url{https://gap-packages.github.io/smallgrp/}

\bibitem[BGSVW15]{BardakovEtAl2015}
V.~G. Bardakov, K.~Gongopadhyay, M.~Singh, A.~Vesnin, and J.~Wu,
\emph{Some problems on knots, braids, and automorphism groups},
Siberian Electron. Math. Rep. \textbf{12} (2015), 394--405.
\url{https://doi.org/10.17377/semi.2015.12.033}

\bibitem[BE16]{BenEzra2016}
D.~El-Chai Ben-Ezra,
\emph{The congruence subgroup problem for the free metabelian group on two
generators},
Groups Geom. Dyn. \textbf{10} (2016), no.~2, 583--599.
\url{https://doi.org/10.4171/GGD/357}

\bibitem[CT17]{CornulierTessera2017}
Y.~Cornulier and R.~Tessera,
\emph{Geometric presentations of Lie groups and their Dehn functions},
Publ. Math. Inst. Hautes \`Etudes Sci. \textbf{125} (2017), 79--219.
\url{https://doi.org/10.1007/s10240-016-0087-3}

\bibitem[DS96]{DarkScoppola1996}
R.~Dark and C.~M. Scoppola,
\emph{On Camina groups of prime power order},
J. Algebra \textbf{181} (1996), no.~3, 787--802.
\url{https://doi.org/10.1006/jabr.1996.0146}

\bibitem[F53]{Fox1953}
R.~H. Fox,
\emph{Free differential calculus. I. Derivation in the free group ring},
Ann. of Math. (2) \textbf{57} (1953), 547--560.
\url{https://doi.org/10.2307/1969736}

\bibitem[GAP]{GAP4}
The GAP Group,
\emph{GAP---Groups, Algorithms, and Programming}, Version~4.15.1, 2025.
\url{https://www.gap-system.org/}

\bibitem[GT97]{GuptaTimoshenko1997}
C.~K. Gupta and E.~I. Timoshenko,
\emph{Automorphic and endomorphic reducibility and primitive endomorphisms of
free metabelian groups},
Comm. Algebra \textbf{25} (1997), no.~10, 3057--3070.
\url{https://doi.org/10.1080/00927879708826040}

\bibitem[HS62]{HulanickiSwierczkowski1962}
A.~Hulanicki and S.~\'{S}wierczkowski,
\emph{On group operations other than \(xy\) or \(yx\)},
Publ. Math. Debrecen \textbf{9} (1962), 142--148.
\url{https://doi.org/10.5486/PMD.1962.9.1-2.15}

\bibitem[KM26]{kourovka21}
E.~I. Khukhro and V.~D. Mazurov (eds.),
\emph{Unsolved Problems in Group Theory: The Kourovka Notebook},
No.~21, Novosibirsk, 2026.
\url{https://kourovkanotebook.org/}

\bibitem[L02]{Lee2002}
D.~Lee,
\emph{Primitivity preserving endomorphisms of free groups},
Comm. Algebra \textbf{30} (2002), no.~4, 1921--1947.
\url{https://doi.org/10.1081/AGB-120013224}

\bibitem[LNB25]{LintonNybergBrodda2025}
M.~Linton and C.-F. Nyberg-Brodda,
\emph{The theory of one-relator groups: history and recent progress},
arXiv:2501.18306, 2025.
\url{https://arxiv.org/abs/2501.18306}

\bibitem[LW22]{LouderWilton2022}
L.~Louder and H.~Wilton,
\emph{Negative immersions for one-relator groups},
Duke Math. J. \textbf{171} (2022), no.~3, 547--594.
\url{https://doi.org/10.1215/00127094-2021-0024}

\bibitem[LW24]{LouderWilton2024}
L.~Louder and H.~Wilton,
\emph{Uniform negative immersions and the coherence of one-relator groups},
Invent. Math. \textbf{236} (2024), no.~2, 673--712.
\url{https://doi.org/10.1007/s00222-024-01246-4}

\bibitem[M49]{Malcev1949}
A.~I. Mal'cev,
\emph{Nilpotent torsion-free groups},
Izv. Akad. Nauk SSSR Ser. Mat. \textbf{13} (1949), no.~3, 201--212
(in Russian).
\url{https://www.mathnet.ru/eng/im3184}

\bibitem[MG13]{MuktibodhGhate2013}
A.~S. Muktibodh and S.~H. Ghate,
\emph{On Camina group and its generalizations},
Matemati\v{c}ki Vesnik \textbf{65} (2013), no.~2, 250--260.
\url{https://www.vesnik.math.rs/vol/mv13212.pdf}

\bibitem[MKS76]{MKS1976}
W.~Magnus, A.~Karrass, and D.~Solitar,
\emph{Combinatorial Group Theory}, 2nd ed.,
Dover, New York, 1976.

\bibitem[N61]{Neumann1961}
H.~Neumann,
\emph{On a question of Kert\'esz},
Publ. Math. Debrecen \textbf{8} (1961), 75--78.
\url{https://doi.org/10.5486/PMD.1961.8.1-2.06}

\bibitem[P77]{Plotkin1977}
B.~I. Plotkin,
\emph{Varieties of group representations},
Russian Math. Surveys \textbf{32} (1977), no.~5, 1--72.
\url{https://doi.org/10.1070/RM1977v032n05ABEH003867}

\bibitem[PV83]{PlotkinVovsi1983}
B.~I. Plotkin and S.~M. Vovsi,
\emph{Varieties of Group Representations: General Theory, Connections and
Applications},
Zinatne, Riga, 1983, 338 pp. (in Russian).

\bibitem[ABEMN25]{SONATA}
E.~Aichinger, F.~Binder, J.~Ecker, P.~Mayr, and C.~N\"obauer,
\emph{SONATA: System of Nearrings and Their Applications}, Version~2.9.7,
GAP package, 2025.
\url{https://gap-packages.github.io/sonata/}

\bibitem[S26]{Semidetnov2026}
A.~Semidetnov,
\emph{The stable commutator length of a relator is not a one-relator group
invariant},
arXiv:2608.21465, 2026.
\url{https://arxiv.org/abs/2608.21465}

\bibitem[T13]{Timoshenko2013}
E.~I. Timoshenko,
\emph{Systems of elements preserving measure on varieties of groups},
Sb. Math. \textbf{204} (2013), no.~12, 1811--1818.
\url{https://doi.org/10.1070/SM2013v204n12ABEH004361}

\bibitem[T15]{Timoshenko2015}
E.~I. Timoshenko,
\emph{Endomorphisms of free solvable groups preserving primitivity of systems
of elements},
Algebra Logic \textbf{54} (2015), no.~4, 323--335.
\url{https://doi.org/10.1007/s10469-015-9352-7}

\bibitem[TW80]{ThomasWood1980}
A.~D. Thomas and G.~V. Wood,
\emph{Group Tables}, Shiva Publishing, Nantwich, 1980.

\bibitem[TS26]{TimoshenkoShpilrain2026}
E.~I. Timoshenko and V.~Shpilrain,
\emph{Problem (M0)}, in \emph{Metabelian Groups},
The World of Group Theory, online problem list (accessed August~28, 2026).
\url{https://shpilrain.ccny.cuny.edu/gworld/problems/probmet.html}

\bibitem[WdN95]{vanderWaallDeNijs1995}
R.~W. van der Waall and C.~H.~W.~M. de Nijs,
\emph{On the embedding of a finite group as Frattini subgroup},
Bull. Belg. Math. Soc. Simon Stevin \textbf{2} (1995), 519--527.
\url{https://eudml.org/doc/231211}

\bibitem[DJMM26]{vandoorn-et-al}
W.~van Doorn, E.~Judin, P.~Monticone, and D.~Morrison,
\emph{On Some Problems from the Kourovka Notebook},
arXiv:2607.17477, 2026.
\url{https://arxiv.org/abs/2607.17477}

\end{thebibliography}
\end{document}